\documentclass{amsart}

\usepackage{amssymb}
\usepackage{picture}
\usepackage{color}

\newtheorem{theorem}{Theorem}[section]
\newtheorem{lemma}[theorem]{Lemma}

\newtheorem{problem}[theorem]{Problem}

\newtheorem{notation}[theorem]{Notation}
\newtheorem{definition}[theorem]{Definition}
\newtheorem{remark}[theorem]{Remark}

\newcommand{\MO}{\mathcal O}
\newcommand{\MP}{\mathcal P}
\newcommand{\MW}{\mathcal W}
\newcommand{\N}{\mathbb N}

\newcommand{\R}{\mathbb R}

\newcommand{\IP}{\mathbb P}
\newcommand{\E}{\mathbb E}
\newcommand{\bx}{\boldsymbol{x}}
\newcommand{\by}{\boldsymbol{y}}
\newcommand{\bz}{\boldsymbol{z}}
\newcommand{\ba}{\boldsymbol{\alpha}}

\newcommand{\on}{\operatorname}
\newcommand{\card}{\operatorname{card}}
\author{Jacek Marchwicki}
\address{Institute of Mathematics, \L \'od\'z University of Technology,
W\'olcza\'nska 215, 93-005 \L \'od\'z, Poland}
\email {marchewajaclaw@gmail.com}

\author{V\' aclav Vlas\'ak}
\address{Faculty of Mathematics and Physics, Charles University, Sokolovsk\'a 83, 18675 Praha~8, Czech Republic}
\email{vlasakvv@gmail.com}

\title[Subsums of conditionally convergent series in finite dimensional spaces]{Subsums of conditionally convergent series in finite dimensional spaces}

\subjclass[2010]{Primary: 40A05 ; Secondary: 11K31} 
\keywords{achievement set, set of subsums, conditionally convergent series, sum range, prime numbers}

\begin{document}

\begin{abstract}
An achievement set of a series is a set of all its subsums. We study the properties of achievement sets of conditionally convergent series in finite dimensional spaces. The purpose of the paper is to answer some of the open problems formulated in \cite{GM}. We obtain general results for series with harmonic-like coordinates, that is $\on{A}((-1)^{n+1}n^{-\alpha_1},\dots,(-1)^{n+1}n^{-\alpha_d})=\R^d$ for pairwise distinct numbers $\alpha_1,\dots,\alpha_d\in(0,1]$. For $d=2$, $\alpha_1=1, \alpha_2=\frac{1}{2}$ it was stated as an open problem in  \cite{GM}, that is $\on{A}(\frac{(-1)^n}{n},\frac{(-1)^n}{\sqrt{n}})=\R^2$.
\end{abstract} 

\maketitle

\section{Introduction}

For a sequence $(x_{n})$ (or a series $\sum_{n=1}^\infty x_n$) we call the set $\on{A}(x_{n})=\{\sum_{n=1}^\infty\varepsilon_nx_n : (\varepsilon_{n})\in\{0,1\}^{\mathbb{N}}\}$ \emph{the set of subsums} or \emph{the achievement set}. 
This notion was mostly studied for absolutely summable sequences on the real line. Probably the first paper where topological properties of achievement sets were investigated is that of Kakeya \cite{Kakeya}.
He proved that such sets can be:
\begin{itemize}
\item finite sets,
\item finite unions of compact intervals (if $\vert x_k\vert\leq\sum_{n=k+1}^\infty \vert x_n\vert$ for almost all $k$ and it is a single interval if the inequality holds for all $k$),
\item homeomorphic to the Cantor set (if $\vert x_k\vert>\sum_{n=k+1}^\infty \vert x_n\vert$ for almost all $k$ - so called quickly convergent series $\sum_{n=1}^\infty x_n$).
\end{itemize}

Kakeya conjectured that Cantor-like sets and finite unions of closed intervals are the only possible achievement sets for sequences $(x_{n})\in \ell_{1}\setminus c_{00}$. The results of Kakeya were rediscovered many times and his conjecture was repeated, even after the first counterexamples were given. In 1970 Renyi in \cite{R} repeated the thesis of Kakeya Theorem in terms of purely atomic measures and he asked if the Cantor-like sets and finite unions of closed intervals are the only possible sets being the ranges of finite measures. Geometric properties of achievement sets of absolutely summable sequences and ranges of purely atomic $\sigma$-finite measures are the same. This follows from the simple observation, that the set of sums of subseries for the series $\sum_{n=1}^{\infty} x_n$ is isometric to the analogous set for the series of their absolute values $\sum_{n=1}^{\infty} \vert x_n\vert$. Therefore a positive answer for the Renyi's question is equivalent to the Kekeya's conjecture. However the first counterexamples were published by Weinstein and Shapiro \cite{WS}, Ferens \cite{Ferens} and Guthrie and Nymann \cite{GN}. It is worth to mention that the motivation of Ferens' paper \cite{Ferens} came from measure theory; namely, the Author constructed purely atomic probabilistic measure the range of which is neither finite union of intervals nor homeomorphic to the Cantor set. Due to Guthrie, Nymann and Saenz \cite{GN,NS1} we know that the achievement set of an absolutely summable sequence can be a finite set, a finite union of intervals, homeomorphic to the Cantor set or it can be a so-called Cantorval. A Cantorval is a set homeomorphic to the union of the Cantor set and sets which are removed from the unit segment by even steps of the construction of the Cantor set.  It is known that a Cantorval is such nonempty compact
set in $\mathbb{R}$, that it is the closure of its interior and both
endpoints of any nontrivial component are accumulation points of its trivial
components. Other topological characterizations of Cantorvals can be found
in \cite{BFPW} and \cite{MO}.
All known examples of sequences whose achievement sets are
Cantorvals belong to the class of multigeometric sequences or are linear combinations of such sequences, see \cite{BBGS},\cite{BBGS1}. This class was
deeply investigated in \cite{Jones}, \cite{BFS} and \cite{BBFS}. In particular,
the achievement sets of multigeometric series and similar sets obtained in more
general case are the attractors of the affine iterated function systems, see 
\cite{BBFS}. More information on achievement sets can be found in the
surveys \cite{BFPW}, \cite{N1} and \cite{N2}.

%In 1970 Renyi in \cite{R} repeated the thesis of Kakeya Theorem in terms of purely atomic measures and he asked if the Cantor-like sets and finite unions of closed intervals are the only possible sets being the ranges of finite measures. Geometric properties of achievement sets of sequences and ranges of purely atomic $\sigma$-finite measures are the same. This follows from the simple observation, that the set of sums of subseries for the series $\sum_{n=1}^{\infty} x_n$ is isometric to the analogous set for the series of their absolute values $\sum_{n=1}^{\infty} \vert x_n\vert$. Therefore a positive answer for the Renyi's question is equivalent to the Kekeya's conjecture.

Achievement sets can also be considered for conditionally convergent series but then we take only those $(\varepsilon_{n})\in\{0,1\}^{\mathbb{N}}$ for which $\sum_{n=1}^\infty\varepsilon_nx_n$ converges.
By $\on{SR}(x_n)=\{\sum_{n=1}^\infty x_{\sigma(n)} : \sigma\in S_{\infty}\}$ we denote the sum range of a series $\sum_{n=1}^\infty x_n$.  If $\sum_{n=1}^\infty x_n$ is conditionally convergent in $\R^m$, then the classical Levy-Steinitz Theorem states that $\on{SR}(x_n)$ is an affine subset of the underlying space. More precisely, $\on{SR}(x_n)=\sum_{n=1}^\infty x_n+\Gamma^{\perp}$ where $\Gamma^{\perp}$ is a subspace orthogonal to the set $\Gamma=\{f\in(\R^m)^*:\sum_{n=1}^\infty\vert f(x_n)\vert<\infty\}$ of all functionals of series convergence. The theory of rearrangements of conditionally convergent series in Banach spaces, and further in topological vector spaces, has been developed and deeply investigated by many authors; we refer the reader to the monograph \cite{KadKad} for details. 
%. He proved that such sets can be finite sets, finite unions of compact intervals or homeomorphic to the Cantor set. His conjecture that the achievement set can be only one of these three forms was disproved by Weinstein and Shapiro \cite{WS}, Ferens \cite{Ferens} and Guthrie and Nymann \cite{GN}.  Topological classification of achievement sets on the real line was given by Guthrie, Nymann and Saenz \cite{GN,NS1} who proved that they can be a finite set, a finite union of intervals, homeomorphic to the Cantor set or it can be a so called Cantorval (see also \cite{BFPW}). A Cantorval is a set homeomorphic to the union of the Cantor set and sets which are removed from the unit segment by even steps of the Cantor set construction. If underlying sequence is regular, for example multigeometric, then achievement sets are affine IFS fractals. Fractal geometry of achievement sets were studied in \cite{BBFS,BG,BFS,MO,M1,M2}. 
%The achievement set of conditionally convergent series of real numbers is the whole real line \cite{BFPW,Jones,N1,N2}. Studies of achievement sets of conditionally convergent series in multidimensional spaces were initiated by us in \cite{BGM}. %It turns out that theory of such sets is connected with theory of series' rearrangements. 
In \cite{BGM} the authors focused mostly on the case when $\sum_{n=1}^\infty x_n$ is conditionally convergent in $\R^2$ and $\on{SR}(x_n)$ is a line. 
They showed that $\on{A}(x_{n})$ can essentially differ from $\on{SR}(x_n)$, in particular when the sum range is one dimensional, affine subspace of  $\R^2$ then
it is possible to obtain the achievement set, which is: not closed; a graph of function; neither an $F_{\sigma}$ nor a $G_{\delta}$-set; open set differs from $\R^2$.
They made a general observation that $\on{SR}(x_n)=\R^m$ if and only if the closure of $\on{A}(x_{n})$ equals $\R^m$ as well. The authors also constructed an example of series on the plane such that $\on{SR}(x_n)=\R^2$ and $\on{A}(x_{n})$ is dense and null. The counterpart of this example constructed in \cite{GM} shows that it is also possible to obtain $\on{A}(x_n)$ as a graph of a partial function, when $\on{SR}(x_n)=\R^2$.
 
A partial answer to the question what needs to be assumed on the series $\sum_{n=1}^\infty x_n$ with $\on{SR}(x_n)=\R^2$ to obtain $\on{A}(x_{n})=\R^2$ is given in \cite{GM}. It depends firstly on the number of Levy vectors. A vector $u\in\mathbb{R}^2$, $\Vert u\Vert =1$ is called the Levy vector of a series $\sum_{n=1}^{\infty}v_n$ if for every $\varepsilon>0$ we have $\sum_{v_n\in S_{\varepsilon}(u)} \Vert v_{n}\Vert=\infty$, where  $S_{\varepsilon}(u)=\{v : \langle u,v\rangle \geq (1-\varepsilon) \Vert u\Vert \Vert v \Vert\}$, where by $\langle u,v\rangle$ we denote the scalar product of $u$ and $v$.
The authors showed that if a series has more than two Levy vectors, then $\on{A}(x_n)=\on{SR}(x_n)=\R^2$. 
They proved even more: for any $a\in\R^2$ there is an increasing sequence $(n_k)$ of indexes such that $\sum_{k=1}^\infty x_{n_k}$ is absolutely convergent to $a$. In symbols: $\on{A}_{\on{abs}}(x_{n})=\R^2$, where $\on{A}_{\on{abs}}(x_{n})=\{\sum_{n=1}^{\infty} \varepsilon_{n} x_{n}  : \sum_{n=1}^{\infty} \varepsilon_{n}\Vert x_{n}  \Vert<\infty,  \varepsilon_{n}\in\{0,1\} \ \text{for each} \  n\in\mathbb{N} \}$.
 The authors also found a necessary condition (\emph{reduction property}) for $\on{A}(x_n)=\R^2$ for a series with  exactly two Levy vectors. %for series $\sum_{n=1}^\infty x_n$. with $\vert \on{L}(x_n)\vert=2$. 
They constructed an example of series $\sum_{n=1}^\infty x_n$ with two Levy vectors and such that $\on{A}_{\on{abs}}(x_n)\neq \on{A}(x_n)=\R^2$. 
At the end of \cite{GM} the authors proposed some open problems. One of them was to check if the equality $\on{A}(\frac{(-1)^n}{n},\frac{(-1)^n}{\sqrt{n}})=\R^2$ holds. The series $\sum\limits_{n=1}^{\infty}(\frac{(-1)^n}{n},\frac{(-1)^n}{\sqrt{n}})$ is very problematic since it has two Levy vectors and it is not known if it satisfies the reduction property. In this paper we give a positive answer to that question, that is $\on{A}(\frac{(-1)^n}{n},\frac{(-1)^n}{\sqrt{n}})=\R^2$. We obtain something more general, that is $\on{A}_{\on{abs}}(\frac{(-1)^n}{n},\frac{(-1)^n}{n^{\alpha}})=\R^2$ for each $\alpha\in (0,1)$. We study its generalization in higher dimensions.

\section{Main result}
 In this chapter we often represent a set $A\subset\N$ as $\{a_n;\ n\in\N\}$. In those cases, we assume that this sequence is increasing.

\begin{notation} Let $d\in\N$, $A\subset\N$, $\bx=(x_1,\dots,x_d)\in\R^d$, $\ba=(\alpha_1,\dots,\alpha_d)\in(0,1]^d$ and $\alpha,\delta\in(0,1]$. Then we can apply some useful notions.

\begin{itemize}

\item{} $\IP=\{n\in\N;\ n\text{ is odd prime number}\},$
\item{} $<\bx,\ba>=(x^1\alpha^1,\dots,x^d\alpha^d)$,
\item{} $\E_{\delta}=\{\{a_n;\ n\in\N\}\subset\N;\ (\forall n\in\N:\ a_{n+1}\leq a_n(1+\delta))\wedge (\exists\epsilon>0\exists n_0\in\N\forall n\geq n_0:\ a_{n+1}\geq a_n(1+\epsilon))\}$,
\item{} $\MP=\{A\subset\IP;\ (\sum_{n\in A}\frac{1}{n}=\infty)\wedge (\forall\delta\in(0,1]\exists A_{\delta}\subset A:\ A_{\delta}\in\E_{\delta}\}$,
\item{} $\MO^d=\{((x^i)_{i=1}^d,(y^i)_{i=1}^d)\in\R^{2d};\ \forall i\in\{1,\dots,d\}:(x^iy^i<0)\vee(x^i=y^i=0)\}$,
\item{} $\Phi(A,\alpha)=\sum_{n\in A}n^{-\alpha}(-1)^{n+1}$, if series is absolutely convergent or $A\subset(2\N+1)$,
\item{} $\Psi(A,\ba)=(\Phi(A,\alpha_1),\dots,\Phi(A,\alpha_d))$, if series is absolutely convergent or $A\subset(2\N+1)$,
\item{} $A_{|_d}=A\cap[d+1,+\infty)$.

\end{itemize}

\end{notation}

\begin{definition}

For $i\in\N$ we inductively define collections $\MW_i$ and $\MW$ of subsets of $\N$. We put
\begin{eqnarray}
\MW_{1}&=&\{A\subset\N;\ \exists B\in\E_1:\ A\subset B\},\nonumber\\
\MW_{i+1}&=&\{A\cdot B\cup C;\ A\in\MW_1\wedge B,C\in\MW_i\},\ i\in\N\nonumber\\
\MW&=&\bigcup_{i=1}^{\infty}\MW_i.\nonumber
\end{eqnarray}

\end{definition}

\begin{definition}

Let $A\subset\N$, $B\subset\IP$ and $p\in B$. We say that $A$ is constructed from $(B,p)$ if every element of $A$ is not divisible by any element of $\IP\setminus B$ and is divisible by $p$.

\end{definition}

In the next lemma we show that for $A\in\MW$, the series $\Phi(A,\alpha)$ is absolutely convergent. Thus, the ordering of $A$ is not important and the notion $\Phi(A,\alpha)$ is well defined. If we assume that $A$ is a subset of odd numbers, then all terms of $\Phi(A,\alpha)$ are positive, so $\Phi(A,\alpha)$ is also well defined. In the following paper we will use $\Phi(A,\alpha)$ only for $A\in\MW$ or $A\subset\IP$. 

\begin{lemma} \label{base1}

Let $\alpha\in(0,1]$.

\begin{itemize}

\item[(A)] Let $A,B\subset\N$. Then $\sum_{n\in A\cdot B}(n)^{-\alpha}\leq\sum_{n\in A}(n)^{-\alpha}\sum_{n\in B}(n)^{-\alpha}$ and the equality holds if the mapping $(a,b)\to a\cdot b$ is injective on $A\times B$.
\item[(B)] Let $A\in\E_1$. Then $\sum_{n\in A}n^{-\alpha}<+\infty$.
\item[(C)] Let $A\in\MW$. Then $\sum_{n\in A}n^{-\alpha}<+\infty$.

\end{itemize}

\end{lemma}

\begin{proof}

Proposition (A) simply follows from the fact that $a^{-\alpha}b^{-\alpha}=(ab)^{-\alpha}$.

Now, we prove proposition (B). Let $A=\{a_j;\ j\in\N\}\in\E_1$. Then there exist $\epsilon>0$, $j_0\in\N$ such that for every $j\geq j_0$ we have $a_{j+1}\geq a_j(1+\epsilon)$. Clearly,
$$\sum_{n\in A}n^{-\alpha}=\sum_{j=1}^{j_0-1}(a_j)^{-\alpha}+\sum_{j=j_0}^{\infty}(a_j)^{-\alpha}\leq\sum_{j=1}^{j_0-1}(a_j)^{-\alpha}+(a_{j_0})^{-\alpha}\sum_{n=0}^{\infty}(1+\epsilon)^{-n\alpha}=\sum_{j=1}^{j_0-1}(a_j)^{-\alpha}+\frac{(a_{j_0})^{-\alpha}}{1-(1+\epsilon)^{-\alpha}}<+\infty.$$

To prove proposition (C), we need to show that for every $k\in\N$ and every $A\in\MW_k$ we have
\begin{eqnarray}
\sum_{n\in A}n^{-\alpha}<+\infty.\label{abskonv}
\end{eqnarray}
We will prove this by induction. The case when $k=1$ immediately follows from (B). Assume that $A\in\MW_{k+1}$ and we already proved $(\ref{abskonv})$ for any $C\in\MW_k$. Then $A=B\cdot C\cup D$, where $B\in\MW_1$ and $C,D\in\MW_k$. By propositions (A) we simply obtain
$$\sum_{n\in A}n^{-\alpha}\leq\sum_{n\in B}n^{-\alpha}\sum_{n\in C}n^{-\alpha}+\sum_{n\in D}n^{-\alpha}<+\infty,$$
which proves $(\ref{abskonv})$.

\end{proof}

Following remark demonstrates some simple properties of the above defined notions.

\begin{remark} \label{help}
The following assertions hold:
\begin{itemize}

\item[(i)] $\IP\in\MP$.
\item[(ii)] Let $A\in\MP$ and $k\in\N$ then $A_{|_k}\in\MP$.
\item[(iii)] Let $A\subset\IP$ then $(-1)^{n+1}=1$ for any $n\in A$ and $\Phi(A,\alpha)=\sum_{n\in A}n^{-\alpha}$.
\item[(iv)] Let $\alpha\in(0,1]$ and $A,B,C\in\MW$, $A\subset\IP$ and $p\in\IP\setminus A$ be such that $B\cup C$ is constructed from $(\IP\setminus A,p)$. Then 
$$\Phi((A\cdot B)\cup C,\alpha)=\Phi(A,\alpha)\Phi(B,\alpha)+\Phi(C,\alpha).$$
\item[(v)] Let $A\in\MP$ and $k\in\N$. Then there exists $B,C\in\MP$ such that $B\cup C\subset A_{|_k}$ and $B\cap C=\emptyset$.
\item[(vi)] Let $\alpha\in(0,1]$ and $\tilde{B}\in\E_1$. Then $\{\Phi(C,\alpha);\ C\subset \tilde{B}\}=[0,\Phi(\tilde{B},\alpha)]$.
\item[(vii)] Let $x>0$, $\alpha\in(0,1]$ and $A\in\MP$. Then there exists $B\subset A$ such that $B\in\MW_1$ and $\Phi(B,\alpha)=x$.
\end{itemize}

\end{remark}

\begin{proof}

It is well known that $\sum_{p\in\IP}\frac{1}{p}=+\infty$, which was proved by Euler in \cite{Euler}. For $x>0$ we define 
\begin{eqnarray}\nonumber
f(x)&=&\card(\IP\cap[0,x]),\\
g(x)&=&\frac{f(x)\log(x)}{x}\nonumber.
\end{eqnarray}
Prime Number Theorem states that
\begin{eqnarray}\label{density}
\lim_{x\to\infty}g(x)=1. 
\end{eqnarray}
Now, we prove that for every $0<\epsilon<\delta\leq1$, there exists $n_0\in\N$ such that for every $n\geq n_0$ we have 
\begin{eqnarray}\label{Adelta}
\IP\cap[n(1+\epsilon),n(1+\delta)]\neq\emptyset.
\end{eqnarray}
Assume on the contrary that there exist $0<\epsilon<\delta\leq1$ for which there exists an increasing sequence $\{a_n;\ n\in\N\}$ of integers such that
$$f(a_n(1+\epsilon))=f(a_n(1+\delta)).$$
By $(\ref{density})$ we have
$$
1=\lim_{n\to\infty}\frac{g(a_n(1+\epsilon))}{g(a_n(1+\delta))}=\lim_{n\to\infty}\frac{f(a_n(1+\epsilon))\log(a_n(1+\epsilon))a_n(1+\delta)}{f(a_n(1+\delta))\log(a_n(1+\delta))a_n(1+\epsilon)}=\lim_{n\to\infty}\frac{\log(a_n(1+\epsilon))(1+\delta)}{\log(a_n(1+\delta))(1+\epsilon)}=\frac{1+\epsilon}{1+\delta},
$$
which is a contradiction. Let $\delta\in(0,1]$ be arbitrary. By $(\ref{Adelta})$ we can find $A_{\delta}\subset\IP$ such that $A_{\delta}\in\E_{\delta}$. So, we proved proposition (i).

Proposition (ii) is trivial.

Proposition (iii) is trivial.

Proposition (iv) simply follows from Lemma~\ref{base1}(A). We only need to show that $(A\cdot B)\cap C=\emptyset$ and the mapping $(a,b)\to a\cdot b$ is injective on $A\times B$. This immediately follows from the fact that any element of $B\cup C$ is not divisible by any element of $A$.

Now, we prove proposition (v). Let $A_{|_k}=\{a_n;\ n\in\N\}$. Put $B=\{a_{2n-1};\ n\in\N\}$ and $C=\{a_{2n};\ n\in\N\}$. Clearly 
$$\sum_{n\in B}\frac{1}{n}=\sum_{n\in C}\frac{1}{n}=+\infty.$$
Let $\delta\in(0,1]$. We need to find $B_{\delta}\subset B$ such that $B_{\delta}\in\E_{\delta}$. By the definition of $\MP$ there exists an increasing sequence of integers $\{n_k\}_{k=1}^{\infty}$ such that $\{a_{n_k};\ k\in\N\}\in\E_{\frac{\delta}{7}}$. Thus there exists $\epsilon>0$ and $k_0\in\N$ such that for every $k\geq k_0$ we have $a_{n_{k+1}}\geq(1+\epsilon)a_{n_k}$ and for every $k\in\N$ we have $a_{n_{k+1}}\leq(1+\frac{\delta}{7})a_{n_k}$. We define a set $B_{\delta}=\{x_i;\ i\in\N\}$ by
$$x_i=
\begin{cases}
a_{n_{2i}}:\ n_{2i}\text{ is odd,}\\
a_{n_{2i}-1}:\ n_{2i}\text{ is even.}
\end{cases}
$$
Clearly, $B_{\delta}\subset B$. So, it remains to be shown that $B_{\delta}\in\E_{\delta}$. Let $i\geq k_0$, then
$$x_{i+1}\geq a_{n_{2i+2}-1}\geq a_{n_{2i+1}}\geq(1+\epsilon)a_{n_{2i}}\geq(1+\epsilon)x_i.$$
Let $i\in\N$ be arbitrary. Then
$$x_{i+1}\leq a_{n_{2i+2}}\leq \left(1+\frac{\delta}{7}\right)^3a_{n_{2i-1}}\leq (1+\delta)a_{n_{2i-1}}\leq (1+\delta)a_{n_{2i}-1}\leq (1+\delta)x_i.$$
So, $B_{\delta}\in\E_{\delta}$. Similarly, we can find $C_{\delta}\subset C$ such that $C_{\delta}\in\E_{\delta}$.

Assume that $\tilde{B}=\{a_n;\ n\in\N\}$. Proposition (vi) follows from the fact that the terms $(a_n)^{-\alpha}$ tend to $0$ and $(a_{n+1})^{\alpha}\leq2(a_n)^{\alpha}$.

Finally, we prove proposition (vii). Since $A\in\MP$, we can find $C\subset A$ such that $C\in\E_1$. By Lemma~\ref{base1}(B), there exists $y>0$ such that $\Phi(C,\alpha)=y$. Since $A\in\MP$, we have $\Phi(A,1)=+\infty$. Thus
$$\Phi(A\setminus C,\alpha)=\Phi(A,\alpha)-\Phi(C,\alpha)\geq \Phi(A,1)-\Phi(C,\alpha)=+\infty-y=+\infty.$$
Since $n^{-\alpha}$ tends to $0$, we can find a finite set $D\subset (A\setminus C)$ such that $\Phi(D,\alpha)\in[x-y,x)$. Put $\tilde{B}=C\cup D$. Clearly $\tilde{B}\subset A$, $\tilde{B}\in\E_1$ and $\Phi(\tilde{B},\alpha)=\Phi(C,\alpha)+\Phi(D,\alpha)\geq x$. We use Proposition (vi) to find $B\subset\tilde{B}$ such that $\Phi(B,\alpha)=x$.

\end{proof}

In the following lemma, we prove a stronger version of Remark~\ref{help}(vii).

\begin{lemma}\label{Pom1}

Let $x,y,z>0$, $\tilde{A}\in\MP$, $0<\alpha<\beta<\gamma$ and $\beta\leq1$. Then there exists $A\subset\tilde{A}$ such that

\begin{eqnarray}
&A&\in\MW_1,\nonumber\\
&\Phi(A,\alpha)&>z,\nonumber\\
&\Phi(A,\beta)&=x,\label{L1,2}\nonumber\\
&\Phi(A,\gamma)&<y.\label{L1,3}\nonumber
\end{eqnarray}

\end{lemma}

\begin{proof}

Let $k\in\N$ be arbitrary. By Remark~\ref{help}(ii) we have $\tilde{A}_{|_k}\in\MP$. By Remark~\ref{help}(vii), there exist $C_k\subset\tilde{A}_{|_k}$ such that $C_k\in\MW_1$ and $\Phi(C_k,\beta)=x$. Clearly,
\begin{eqnarray}
\Phi(C_k,\alpha)=\sum_{n\in C_k}n^{-\alpha}=\sum_{n\in C_k}n^{-\beta}n^{-\alpha+\beta}\geq\sum_{n\in C_k}n^{-\beta}k^{-\alpha+\beta}=\Phi(C_k,\beta)k^{-\alpha+\beta}=xk^{-\alpha+\beta},\nonumber\\
\Phi(C_k,\gamma)=\sum_{n\in C_k}n^{-\gamma}=\sum_{n\in C_k}n^{-\beta}n^{-\gamma+\beta}\leq\sum_{n\in C_k}n^{-\beta}k^{-\gamma+\beta}=\Phi(C_k,\beta)k^{-\gamma+\beta}=xk^{-\gamma+\beta}.\nonumber
\end{eqnarray}
To finish the proof we only need to find $k\in\N$ such that $xk^{-\alpha+\beta}>z$ and $xk^{-\gamma+\beta}<y$ and set $A=C_k$.

\end{proof}

The following lemma helps us in doing an inductive step in the proof of Lemma~\ref{Main}.

\begin{lemma}\label{inductive lemma}

Let $d\in\N$, $l\in\{1,\dots,d\}$, $\bx=(x^1,\dots,x^d), \by=(y^1,\dots,y^d)\in \MO^d$, $0<\alpha_1<\dots<\alpha_d\leq1$ and $M\in\MP$. Then there exists $A,B,C\subset\N$ and $\bz_1=(z_1^1,\dots,z_1^d),\bz_2=(z_2^1,\dots,z_2^d)\in\R^d$ such that
\begin{itemize}
\item[(a)] $C\in\MP$,
\item[(b)] $A\cup B\cup C\subset M$,
\item[(c)] $A,B,C$ are pairwise disjoint,
\item[(d)] $A,B\in\MW_1$,
\item[(e)] $(\bz_1,\bz_2)\in\MO^d$,
\item[(f)] $z_1^i=0$ if and only if $(x^i=0$ or $i=l)$,
\item[(g)] 
$$\bz_1=<\Psi(A,(\alpha_1,\dots,\alpha_d)),\bx>+\by,$$ 
$$\bz_2=<\Psi(B,(\alpha_1,\dots,\alpha_d)),\by>+\bx.$$
\end{itemize}
\end{lemma}

\begin{proof}

If $x^l=0$, then we can set $A=B=\emptyset$, $C= M$ satisfying (a)-(d). We define $\bz_1$ and $\bz_2$ by (g). Thus, we simply obtain (e), (f).

Assume $x^l\neq0$. Since $(\bx,\by)\in\MO^d$ we have $-\frac{y^l}{x^l}>0$. By Remark~\ref{help}(v), we can find some pairwise disjoint sets $\tilde{A},\tilde{B},C\subset M$ such that $\tilde{A},\tilde{B},C\in\MP$. Thus $(a)$ is satisfied. Now, we use Lemma~\ref{Pom1} to find the set $A$. We set constants in Lemma~\ref{Pom1} in the following way. For $1<l<d$ set $x=-\frac{y^l}{x^l}$, $z=\frac{\max\{|y^i|;\ i\in\{1,\dots,l-1\}\wedge y^i\neq0\}}{\min\{|x^i|;\ i\in\{1,\dots,l-1\}\wedge x^i\neq0\}}$, $y=\frac{\min\{|y^i|;\ i\in\{l+1,\dots,d\}\wedge y^i\neq0\}}{\max\{|x^i|;\ i\in\{l+1,\dots,d\}\wedge x^i\neq0\}}$, $\alpha=\alpha_{l-1}$ $\beta=\alpha_l$, $\gamma=\alpha_{l+1}$. We set $\max\{\emptyset\}=\min\{\emptyset\}=1$. For $l\in\{1,d\}$ we set them analogously but in case $l=1$ we set $z=1$ and $\alpha=\frac{\alpha_1}{2}$ and in case $l=d$ we set $y=1$ and $\gamma=2$. Analogously, we find $B\subset\tilde{B}$, we only interchange $\bx$ and $\by$. 

Since $A,B\in\MW_1$ and Lemma~\ref{base1}(C) we can define $\bz_1$ and $\bz_2$ by (g).

Since $A\subset\tilde{A}$ and $B\subset\tilde{B}$ we have $(b)$ and $(c)$. We already showed (d) in the previous paragraph. 

Let $j\in\{1,\dots,d\}$ be arbitrary. 

If $x^j=0$ then clearly $z^j_1=z^j_2=0$.

Assume $j=l$. Then 
$$z_1^l=\Phi(A,\alpha_l)x_l+y_l=-\frac{y_l}{x_l}x_l+y_l=0.$$
Similarly $z_2^l=0$.

Now assume that $1\leq j<l$ and $x^j\neq0$. Thus 
$$\Phi(A,\alpha_j)\geq \Phi(A,\alpha_{l-1})>\frac{\max\{|y^i|;\ i\in\{1,\dots,l-1\}\wedge y^i\neq0\}}{\min\{|x^i|;\ i\in\{1,\dots,l-1\}\wedge x^i\neq0\}}.$$
So 
$$\Phi(A,\alpha_j)|x^j|>|y^j|.$$
Obviously, $z_1^j=\Phi(A,\alpha_j)x^j+y^j$ has the same sign as $x^j$. Similarly $z_2^j$ has the same sign as $y^j$. Since $x^jy^j<0$ we have $z_1^jz_2^j<0$.

Finally, let $l<j\leq d$ and $x^j\neq0$. Thus
$$\Phi(A,\alpha_j)\leq \Phi(A,\alpha_{l+1})<\frac{\min\{|y^i|;\ i\in\{l+1,\dots,d\}\wedge y^i\neq0\}}{\max\{|x^i|;\ i\in\{l+1,\dots,d\}\wedge x^i\neq0\}}.$$
So 
$$\Phi(A,\alpha_j)|x^j|<|y^j|.$$
Clearly $z_1^j=\Phi(A,\alpha_j)x^j+y^j$ has the same sign as $y^j$. Similarly $z_2^j$ has the same sign as $x^j$. Since $x^jy^j<0$ we have $z_1^jz_2^j<0$. Thus we proved $(e)$, $(f)$ .
\end{proof}

\begin{lemma}\label{Main}

Let $d\in\N$, $k\in\{1,\dots,d\}$, $0<\alpha_1<\dots<\alpha_d\leq1$, $V\in\MP$, $p\in V$ and $x\in\R$. Then there exists $W\in\MW$ such that $W$ is constructed from $(V,p)$ and
\begin{eqnarray}\nonumber
\Phi(W,\alpha_i)=
\begin{cases}
0;\ i\neq k,\\
x;\ i=k.
\end{cases}
\end{eqnarray}
\end{lemma}

\begin{proof}

If $x=0$ then put $W=\emptyset$. 

If $d=1$ then this lemma simply follows from Remark~\ref{help}(vii). If $x>0$, then we find $\tilde{W}\subset V$ such that $\tilde{W}\in\MW$ and $\Phi(W,\alpha_1)=xp^{\alpha}$. Then we set $W=p\cdot\tilde{W}$. If $x<0$, then we find $\tilde{W}\subset V$ such that $\tilde{W}\in\MW$ and $\Phi(W,\alpha_1)=|x|(2p)^{\alpha}$. Then we set $W=(2p)\cdot\tilde{W}$.

Assume that $x\neq0$ and $d\geq 2$. For $i\in\{1,\dots,d\}$ we inductively construct $\bx_i=(x_i^1,\dots,x_i^d),\by_i=(y_i^1,\dots,y_i^d)\in\R^d$, $l_i\in\N$, $M_i\in\MP$ and $W_i^+,W_i^-\in\MW_i$ satisfying 
\begin{itemize}

\item[(1)] $(\bx_i,\by_i)\in\MO^d$, $i\in\{1,\dots,d\}$,
\item[(2)] $x^j_i=0$ if and only if $j\in\{l_1,\dots,l_{i-1}\}$, $i\in\{2,\dots,d\}$,
\item[(3)] $l_i\in\{1,\dots,d\}\setminus(\{k\}\cup\bigcup_{1\leq j<i}\{l_j\})$, $i\in\{1,\dots,d-1\}$,
\item[(4)] $\Phi(W^+_i,\ba)=x_i$ and $\Phi(W^-_i,\ba)=y_i$, $i\in\{1,\dots,d\}$,
\item[(5)] $W^+_i\cup W^-_i$ is constructed from $(V\setminus M_i,p)$, $i\in\{1,\dots,d\}$,
\item[(6)] $M_i\subset V\setminus\{p\}$, $i\in\{1,\dots,d\}$.

\end{itemize}

For $i=1$ we put $W^+_1=\{p\}$, $W^-_1=\{2p\}$ and $M_1=V\setminus\{p\}$. Clearly, $W^+_1,W^-_1\in\MW_1$, $M_1\in\MP$ and the conditions (5) and (6) are satisfied. Then we define $\bx_1$, $\by_1$ by (4). Thus $x_1^j>0$ and $y_1^j<0$ for every $j\in\{1,\dots,d\}$. So, (1) is also satisfied. Finally put $l_1\in\{1,\dots,d\}\setminus\{k\}$ arbitrarily and (3) is also satisfied.

Assume $i<d$ and $\bx_i,\by_i\in\R^d$, $l_i\in\N$, $M_i\in\MP$ and $W_i^+,W_i^-\in\MW_i$ have already been constructed and satisfy (1)-(6). We use $l_i$, $\bx_i$, $\by_i$ and $M_i$ in Lemma~\ref{inductive lemma} and obtain $A,B,C\subset \N$ and $\bz_1,\bz_2\in\R^d$ such that (a)-(g) are satisfied. We put $M_{i+1}=C$. Thus $M_{i+1}\in\MP$ and $M_{i+1}\subset M_i\subset V\setminus\{p\}$. So, (6) is satisfied. We set $W^+_{i+1}=(A\cdot W^+_i)\cup W^-_i$ and $W^-_{i+1}=(B\cdot W^-_i)\cup W^+_i$. Thus $W^+_{i+1},W^-_{i+1}\in\MW_{i+1}$. Since every element of $W^+_i\cup W^-_i$ is divisible by $p$, we obtain that every element of $W^+_{i+1}\cup W^-_{i+1}$ is also divisible by $p$. Since every element of $W^+_i\cup W^-_i$ is not divisible by any element of $M_i$, $M_{i+1}\subset M_i$ and every element of $A\cup B$ is not divisible by any element of $M_{i+1}$ we have that every element of $W^+_{i+1}\cup W^-_{i+1}$ is not divisible by any element of $M_{i+1}$. Thus we have (5). Put $\bx_{i+1}=\bz_1$ and $\by_{i+1}=\bz_2$. Since $W^+_i\cup W^-_i$ is constructed from $(V\setminus M_i,p)$, Remark~\ref{help}(iv) and Lemma~\ref{inductive lemma}(b),(g) we obtain (4). Conditions (1),(2) immediately follow from Lemma~\ref{inductive lemma}(e),(f). If $i+1<d$ we choose some $l_{i+1}$, which satisfies (3). Otherwise, we put $l_d=k$. So the construction is finished.

Thus, we constructed $\bx_d,\by_d\in\R^d$, $M_d\in\MP$ and $W_d^+,W_d^-\in\MW_i$ satisfying (1),(2),(4),(5),(6). Without loss of generality, we can assume that the sign of $x$ is the same as the sign of $x_d^k$. By Remark~\ref{help}(vii), we can find set $K\subset M_d$ such that $\Phi(K,\alpha_k)=\frac{x}{x_d^k}$ and $K\in\MW_1$. Set $W=K\cdot W^+_d$. By (5) and (6) we obtain that $W$ is constructed from $(V,p)$. Clearly $W\in\MW_{d+1}\subset\MW$. By (2),(4), (5) and Remark~\ref{help}(iv) we obtain 
\begin{eqnarray}\nonumber
\Phi(W,\alpha_i)=
\begin{cases}
0;\ i\neq k,\\
x;\ i=k,
\end{cases}
\end{eqnarray}
and the proof is finished.

\end{proof}

\begin{theorem}

Let $d\in\N$, $\alpha_1,\dots,\alpha_d\in(0,1]$ and $a_1,\dots,a_d\in\R\setminus\{0\}$. Then 
$$\on{A}_{\on{abs}}(a_1(-1)^{n+1}n^{-\alpha_1},\dots,a_d(-1)^{n+1}n^{-\alpha_d})=\R^d$$
if and only if $\alpha_i,\ i=1,\dots,d$ are pairwise distinct numbers.

\end{theorem}

\begin{proof}
If there exist $i\neq j$ such that $\alpha_i=\alpha_j$, then clearly $\on{A}_{\on{abs}}((-1)^{n+1}n^{-\alpha_1},\dots,(-1)^{n+1}n^{-\alpha_d})\neq\R^d$. 

Assume that $\alpha_i,\ i=1,\dots,d$ are pairwise distinct numbers. Clearly, we can assume that $0<\alpha_1<\dots<\alpha_d\leq 1$ and $a_1=\dots=a_d=1$. Let $(x^1,\dots,x^d)\in\R^d$ be arbitrary. By Remark~\ref{help}(v) we can find pairwise disjoint sets $V^k\in\MP$, $p_k\in V_k$, $k\in\{1,\dots,d\}$. Using Lemma~\ref{Main}, we can find $W^k\in\MW$ such that $W^k$ is constructed from $(V^k,p_k)$ and
$$ 
\Phi(W^k,\alpha_i)=
\begin{cases}
0;\ i\neq k,\\
x^k;\ i=k.
\end{cases}
$$
Put $W=\bigcup_{k=1}^dW^k$. Assume that $i\neq j$, $i,j\in\{1,\dots,d\}$. Then any element of $W^i$ is not divisible by $p_j$ and any element of $W^j$ is divisible by $p_j$. Thus, the sets $W^k$, $k\in\{1,\dots,d\}$ are pairwise disjoint. Hence 
$$\Psi(W,(\alpha_1,\dots,\alpha_d))=\sum_{k=1}^d\Psi(W^k,(\alpha_1,\dots,\alpha_d))=(x^1,\dots,x^d)$$ 
and we are done.

\end{proof}

\section{Open problems}
\begin{problem}\label{oneopen}
Characterize a familly of all functions $f$ such that $\on{A}_{\on{abs}}((\frac{(-1)^{n}}{n},\frac{(-1)^{n}}{f(n)})=\R^2$.
\end{problem}
\begin{problem}\label{secondopen}
Characterize a familly of all functions $f$ such that $\on{A}((\frac{(-1)^{n}}{n},\frac{(-1)^{n}}{f(n)})=\R^2$.
\end{problem}
Clearly the familly defined in the first problem is smaller than that defined in the second problem. In the paper we showed that it contains any function $f(n)=n^{\alpha}$ for $\alpha\in(0,1)$. %Simple observation shows that we need to discuss functions $f$ such that $\sum_{n=1}^\infty \frac{1}{f(n)}=\infty$.
%\begin{problem}
%Let $\alpha,\beta\in(0,1), \alpha\neq\beta$ and $(x_n,y_n,z_n)=(\frac{(-1)^{n+1}}{n},\frac{(-1)^{n+1}}{n^{\alpha}},\frac{(-1)^{n+1}}{n^{\beta}})$. Is it true that $\on{A}_{\on{abs}}(x_n,y_n,z_n)=\R^3$.
%\end{problem}


\begin{thebibliography}{a,b,c,d}
\bibitem{BBFS} T. Banakh, A. Bartoszewicz, M. Filipczak, E. Szymonik \emph{%
Topological and measure properties of some self-similar sets}, Topol.
Methods Nonlinear Anal., \textbf{46} (2015), 1013--1028.
\bibitem{BBGS} T. Banakh, A. Bartoszewicz, S. G\l \c{a}b, E. Szymonik,  Algebraic and topological properties of some sets in $\ell_1$, Colloq. Math. 129 (2012), 75--85. 
\bibitem{BBGS1} T. Banakh, A. Bartoszewicz, S. G\l \c{a}b, E. Szymonik, Erratum to ''Algebraic and topological properties of some sets in $\ell_1$'' (Colloq. Math. 129 (2012), 75--85) Colloq. Math. 135 (2014), no. 2, 295--298.
\bibitem{BFPW} A. Bartoszewicz, M. Filipczak, F. Prus-Wi\'sniowski, Topological and algebraic aspects of subsums of series. Traditional and present-day topics in real analysis, 345--366, Faculty of Mathematics and Computer Science. University of \L \'od\'z, \L \'od\'z, 2013.
\bibitem{BFS} A. Bartoszewicz, M. Filipczak, E. Szymonik, Multigeometric sequences and Cantorvals, Cent. Eur. J. Math. 12(7) (2014), 1000--1007.
\bibitem{BGM} A. Bartoszewicz, S. G\l \c{a}b, J. Marchwicki, Achievement sets of conditionally convergent series, arXiv:1604.07575, to appear in  Colloq. Math.
\bibitem{Euler} L. Euler, Variae observationes circa series infinitas, Commentarii Academiae Scientiarum Petropolitanae (1737), 9: 160–188.
\bibitem{Ferens} C. Ferens, On the range of purely atomic probability measures, Studia Math. 77 (1984), 261--263.
\bibitem{GN} J.A. Guthrie, J.E. Neymann, The topological structure of the set of subsums of an infinite series, Colloq. Math. 55:2 (1988), 323--327.
\bibitem{GM} S. G\l \c{a}b, J. Marchwicki, Levy--Steinitz theorem and achievement sets of conditionally convergent series on the real plane, J. Math. Anal. Appl. 459 (2018) 476--489.
\bibitem{Jones} R. Jones, Achievement sets of sequences, Am. Math. Mon. 118:6 (2011), 508--521.
\bibitem{KadKad} M. I. Kadets, V. M. Kadets, Series in Banach spaces. Conditional and unconditional convergence.
\bibitem{Kakeya} S. Kakeya, On the partial sums of an infinite series, T\^{o}hoku Sic. Rep. 3 (1914), 159--164.
\bibitem{MO} P. Mendes, F. Oliveira, On the topological structure of the arithmetic sum of two Cantor sets, Nonlinearity. 7 (1994), 329--343.  
\bibitem{NS1} J.E. Neymann, R.A. S\'{a}enz, On the paper of Guthrie and Nymann on subsums of infinite series, Colloq. Math. 83 (2000), 1--4.  
\bibitem{N1} Z. Nitecki, \emph{The subsum set of a null sequence},
arXiv:1106.3779v1.
\bibitem{N2} Z. Nitecki, \emph{Cantorvals and subsum sets of null sequences}%
, Amer. Math. Monthly, \textbf{122} (2015), no. 9, 862--870.
\bibitem {R} A. R\'enyi, Probability theory, North-Holland, Amsterdam (1970).
\bibitem{WS} A.D. Weinstein, B.E. Shapiro, On the structure of a set of $\overline{\alpha}$-representable numbers, Izv. Vys\v{s}. U\v{c}ebn. Zaved. Matematika. 24 (1980), 8--11. 
\end{thebibliography}
\end{document}